\documentclass[11pt]{article}
\usepackage{amsmath, amsthm, amssymb}
\usepackage[margin=3.5cm]{geometry}
\newtheorem{thm}{Theorem}
\newtheorem{prop}[thm]{Proposition}
\newtheorem{lem}[thm]{Lemma}
\newtheorem{cor}[thm]{Corollary}

\newtheorem{conj}[thm]{Conjecture}
\newtheorem{mydef}[thm]{Definition}
\theoremstyle{remark}
\newtheorem{rem}[thm]{Remark}
\newtheorem{notn}[thm]{Notation}
\newtheorem{assmpt}[thm]{Assumption}
\newcommand{\ord}{\mathrm{ord}}
\newcommand{\term}{\mathrm{term}}
\newcommand{\dege}{\mathrm{dege}}
\newcommand{\orde}{\mathrm{orde}}
\newcommand{\lead}{\mathrm{lead}}
\newcommand{\dgap}{\mathrm{dgap}}
\newcommand{\maxt}{\mathrm{maxt}}
\newcommand{\depth}{\mathrm{depth}}
\title{Identities of the Function $f(x,y) = x^2 + y^3$}
\author{Roger Tian}
\date{\today}
\makeatletter
\newcommand\ackname{Acknowledgements}
\if@titlepage
  \newenvironment{acknowledgements}{%
      \titlepage
      \null\vfil
      \@beginparpenalty\@lowpenalty
      \begin{center}%
        \bfseries \ackname
        \@endparpenalty\@M
      \end{center}}%
     {\par\vfil\null\endtitlepage}
\else
  \newenvironment{acknowledgements}{%
      \if@twocolumn
        \section*{\abstractname}%
      \else
        \small
        \begin{center}%
          {\bfseries \ackname\vspace{-.5em}\vspace{\z@}}%
        \end{center}%
        \quotation
      \fi}
      {\if@twocolumn\else\endquotation\fi}
\fi
\makeatother
\begin{document}
\maketitle
\begin{abstract}
Harvey Friedman asked in $1986$ whether the function $f(x,y) = x^2 + y^3$ on the real plane $\mathbb{R}^2$ satisfies any identities; examples of identities are commutativity and associativity. To solve this problem of Friedman, we must either find a nontrivial identity involving expressions formed by recursively applying $f$ to a set of variables $\{x_1,x_2, \ldots, x_n\}$ that holds in the real numbers or to prove that no such identities hold. In this paper, we will solve certain special cases of Friedman's problem and explore the connection between this problem and certain Diophantine equations.
\end{abstract}
\begin{acknowledgements}
I would like to thank my thesis advisor, George Bergman, for introducing Lemma \ref{single implies multiple} to me, shortening the proofs of several of the results below, pointing out areas of the paper that needed clarification, and giving advice on how to better organize this thesis.
\end{acknowledgements}

Notice that, given any identity in any number of variables, one can get an identity in one variable $x$ by replacing all the variables of the given identity by $x$. The single variable case of Friedman's problem, whether or not there exists a nontrivial identity that holds in the real numbers involving expressions formed by recursively applying $f$ to the variable set $\{x\}$, may be easier to treat than the general problem of multiple variables. 

\textit{A priori}, proving that no nontrivial identity of one variable holds does not completely solve the general problem, because two expressions, if equal as polynomials, that have the same ``structure'' regarding the composition of $f$'s (ignoring the variables involved) lead to the trivial identity when all the variables are replaced by $x$. For instance, $f(f(x,y),f(y,x)) = f(f(y,x),f(x,y))$ as polynomials implies $f(f(x,x),f(x,x)) = f(f(x,x),f(x,x))$ as polynomials. However, proving that no nontrivial identity of one variable holds would tell us that two expressions can be equal as polynomials only if they have the same structure. For instance, since $f(x,f(x,x))$ and $f(f(x,f(x,x)),x)$ are not equal as polynomials, we know that $f(x,f(y,z))$ and $f(f(x,f(y,z)),y)$ cannot be equal as polynomials. We will use this observation to prove in Lemma \ref{single implies multiple} that a nontrivial multiple-variable identity holds only if a nontrivial 1-variable identity holds.

We will follow the convention that $0 \notin \mathbb{N}$.

\begin{notn}
\label{variable position}
Suppose $G(x_1,x_2, \ldots, x_n)$ is an expression formed by recursively applying $f$ to the variable set $\{x_1,x_2, \ldots, x_n\}$. We shall call an occurrence of a variable in $G(x_1,x_2, \ldots, x_n)$ a \textbf{variable position}. Supposing that $G(x_1,x_2, \ldots, x_n)$ contains $l$ variable positions where $l \in \mathbb{N}$, we will proceed from left to right and label these successive variable positions as $v_1$, $v_2$, \ldots, $v_l$, and for all $i = 1, 2, \ldots, l$ we will denote by $\bar{v_i}$ the variable in $\{x_1,x_2, \ldots, x_n\}$ occurring in the variable position $v_i$. We define the \textbf{depth} of a variable position $v$ occurring in an $f$-expression $f(A,B)$ to be one more than its depth in $A$ or $B$ (whichever $v$ occurs in), where we start by defining the depth of the bare expression $x_i$ where $i \in \{1,2, \ldots, n\}$ to be $0$; we will denote the depth of $v$ by $\depth(v)$. For example, the variable positions $v_1$, $v_2$, $v_3$ of $f(x,f(y,x))$ hold the variables $\bar{v_1} = x$, $\bar{v_2} = y$, and $\bar{v_3} = x$, while we have $\depth(v_1) = 1$, $\depth(v_2) = 2$ and $\depth(v_3) = 2$. We can associate to $G(x_1,x_2, \ldots, x_n)$ the $l$-tuple $((\bar{v_1},\depth(v_1)),(\bar{v_2},\depth(v_2)), \ldots ,(\bar{v_l},\depth(v_l)))$. It is clear that $G(x_1,x_2, \ldots, x_n)$  completely determines $((\bar{v_1},\depth(v_1)),(\bar{v_2},\depth(v_2)), \ldots ,(\bar{v_l},\depth(v_l)))$.
\end{notn}

The following result was pointed out to the author by George Bergman, and its proof follows the ideas outlined by Bergman. This lemma establishes that, to answer Friedman's problem in the negative, it suffices to prove that no nontrivial identity of one variable holds.

\begin{lem}
\label{single implies multiple}
Suppose that $f$ satisfies no nontrivial 1-variable identities in $\mathbb{R}$. Let $n$ be a positive integer and let $\{x_1,x_2, \ldots, x_n\}$ be a set of variables. Then $f$ satisfies no nontrivial identities involving the variables $x_1$, $x_2$, \ldots, $x_n$ in $\mathbb{R}$.
\end{lem}
\begin{proof}
Suppose that two distinct expressions $G(x_1,x_2, \ldots, x_n)$, $H(x_1,x_2, \ldots, x_n)$ formed by recursively applying $f$ to $\{x_1,x_2, \ldots, x_n\}$ are equal as polynomials. It follows that $G(x,x, \ldots, x)$ and $H(x,x, \ldots, x)$ are equal as polynomials. Then, by the assumption of no nontrivial 1-variable identities, $G(x,x, \ldots, x)$ and $H(x,x, \ldots, x)$ must be the same expression, so $G(x_1,x_2, \ldots, x_n)$ and $H(x_1,x_2, \ldots, x_n)$ must have the same number of variable positions. We will label the variable positions of $G(x_1,x_2, \ldots, x_n)$ by $v_1$, $v_2$, \ldots, $v_l$ and the variable positions of $H(x_1,x_2, \ldots, x_n)$ by $v_1'$, $v_2'$, \ldots, $v_l'$, where $l$ is some positive integer. Now, $G(x_1,x_2, \ldots, x_n)$ determines the $l$-tuple \[((\bar{v_1},\depth(v_1)),(\bar{v_2},\depth(v_2)), \ldots ,(\bar{v_l},\depth(v_l)))\] and $H(x_1,x_2, \ldots, x_n)$ determines the $l$-tuple \[((\bar{v_1'},\depth(v_1')),(\bar{v_2'},\depth(v_2')), \ldots ,(\bar{v_l'},\depth(v_l'))).\] We know that for each $i = 1, 2, \ldots, l$ we have $\depth(v_i) = \depth(v_i')$, i.e. corresponding variable positions have the same depth. Let $j$ be the smallest positive integer such that $\bar{v_j} = x_k \ne x_m = \bar{v_j'}$ where $k \ne m$. Now replace $x_k$ in $G(x_1,x_2, \ldots, x_n)$, $H(x_1,x_2, \ldots, x_n)$ by $f(x,x)$ and replace $x_i$ in $G(x_1,x_2, \ldots, x_n)$, $H(x_1,x_2, \ldots, x_n)$ by $x$ for each $i \ne k$, and we obtain a 1-variable identity. Then there exists at least one
$p \ge j$ in $\mathbb{N}$ such that the $p$th variable position of the 1-variable expression resulting from $G(x_1,x_2, \ldots, x_n)$ has a depth one greater than that of the $p$th variable position of the 1-variable expression resulting from $H(x_1,x_2, \ldots, x_n)$. Therefore, the two 1-variable expressions in the identity are distinct, which is a contradiction.
\end{proof}
The proof of Lemma \ref{single implies multiple} leaves open a more difficult question, as the statement of Lemma \ref{single implies multiple} is weaker than what we state in the following
\begin{conj}
\label{single multiple conjecture}
Suppose that $G(x)$ is an expression formed by recursively applying $f$ to the variable set $\{x\}$ and that $G(x)$ has the variable positions $v_1$, $v_2$, \ldots, $v_l$ for some positive integer $l$. Let $n$ be a positive integer and $\{x_1,x_2, \ldots, x_n\}$ be a set of variables. Let $G(x_1,x_2, \ldots, x_n)$ be an n-variable expression obtained by letting $\bar{v_i} \in \{x_1,x_2, \ldots, x_n\}$ for all $i = 1, 2, \ldots, l$. If $G(x)$ cannot occur as either side of a nontrivial 1-variable identity, then $G(x_1,x_2, \ldots, x_n)$ cannot occur as either side of a nontrivial $n$-variable identity.
\end{conj}

Below, we prove some results on the single variable case of Friedman's problem. They show that certain classes of expressions cannot occur as either side of a nontrivial identity.
\begin{mydef}
\label{degree order}
If $p(x) = \sum_{k=m}^n{a_kx^k}$ is a polynomial where $m \le n$ are nonnegative integers and where $a_m$, $a_n$ are nonzero, then $m$ will be called the \textbf{order} of $p(x)$, and $n$ will be called the \textbf{degree} of $p(x)$.
\end{mydef}
In what follows, by an \textbf{$f$-expression} we will, unless otherwise specified, always mean a symbolic expression in $f$ and $x$ that is formed by recursively applying $f$ to the variable set $\{x\}$; we also consider $x$ itself an $f$-expression. We will denote the set of all $f$-expressions by $\term(f;x)$. Let $e:\,\, $ $\term(f;x)$ $\, \longrightarrow \, \mathbb{Z}[x]$ be the evaluation map that assigns to each $f$-expression its corresponding polynomial in $\mathbb{Z}[x]$. We say that $e(A)$ is the \textbf{polynomial induced by} the $f$-expression $A$. For example, $e(f(x,f(x,x))) = x^2 + (x^2 + x^3)^3$. If $A$ and $B$ are two $f$-expressions and $e(A) = e(B)$, then we say that $A$ and $B$ are \textbf{$e$-equivalent}. We shall call an $f$-expression \textbf{$e$-isolated} if it is not $e$-equivalent to any other $f$-expression, i.e. it cannot occur as either side of a nontrivial identity. For example, $f(x,x)$ is $e$-isolated because, as it is not hard to see, $e(f(A_1,A_2)) = e(f(x,x))$, where $A_1$, $A_2$ are $f$-expressions, implies $A_1 = x$ and $A_2 = x$.

\begin{notn}
\label{dege orde}
Let $A \in \term(f;x)$. For brevity, we will denote the degree of $e(A)$ by $\dege(A)$ and the order of $e(A)$ by $\orde(A)$. For the degree and order of any polynomial $p(x)$ that is not written as the induced polynomial of some $B \in \term(f;x)$, we retain the standard notation $\deg(p(x))$ and $\ord(p(x))$ respectively. For example, we would denote the degree of $p(x) = x^2 + x^3$ as $\deg(p(x))$ if we did not explicitly state or did not know beforehand that $p(x) = e(f(x,x))$.
\end{notn}

In the next three propositions, let $A, B \in \, \term(f;x)$.

\begin{prop}
\label{appending x left}
If $e(f(C_1,C_2)) = e(f(x,B))$ where $C_1$, $C_2$ are $f$-expressions, then we must have $C_1 = x$ and $e(C_2) = e(B)$.
\end{prop}
\begin{proof}
We have $e(f(x,B)) = x^2 + e(B)^3 = e(C_1)^2 + e(C_2)^3$. Notice that $3\,\orde(C_2) \ge 3$, so $x^2$ must arise in the expansion of $e(C_1)^2$. This forces $C_1 = x$, because otherwise every term that arises in the expansion of $e(C_1)^2$ will have powers at least as high as $4$. Then cancellation gives us $e(C_2)^3 = e(B)^3$, which forces $e(C_2) = e(B)$.
\end{proof}

\begin{prop}
\label{appending x right}
If $e(f(C_1,C_2)) = e(f(A,x))$ where $C_1$, $C_2$ are $f$-expressions, then we must have $e(C_1) = e(A)$ and $C_2 = x$.
\end{prop}
\begin{proof}
We have $e(f(A,x)) = e(A)^2 + x^3 = e(C_1)^2 + e(C_2)^3$. If $C_1 = x$, then $e(C_1)^2 = x^2$. If $C_1 \ne x$, then $\orde(C_1) \ge 2$, so $e(C_1)^2$ has order at least 4. Thus, $x^3$ must arise in the expansion of $e(C_2)^3$. Therefore, we have $C_2 = x$ and it follows from cancellation that $e(C_1)^2 = e(A)^2$, so $e(C_1) = e(A)$.
\end{proof}

\begin{rem}
\label{left right multiple}
Actually, the arguments in the proofs of the previous two propositions also apply if $f(x,B)$ in Proposition \ref{appending x left} and $f(A,x)$ in Proposition \ref{appending x right} are instead multiple-variable $f$-expressions. For example, the same arguments can be applied, repeatedly, to show that $f$-expressions such as $f(f(x,f(y,z)),y)$ are $e$-isolated. In effect, this settles a special case of Conjecture \ref{single multiple conjecture}.
\end{rem}

\begin{prop}
\label{appending f(x,x) left}
If $e(f(C_1,C_2)) = e(f(f(x,x),B))$ where $C_1$, $C_2$ are $f$-expressions, then we must have $C_1 = f(x,x)$ and $e(C_2) = e(B)$.
\end{prop}
\begin{proof}
We have $e(C_1)^2 + e(C_2)^3 = e(f(x,x))^2 + e(B)^3 = (x^2 + x^3)^2 + e(B)^3 = x^4 + 2x^5 + x^6 + e(B)^3$. If $B = x$, then the conclusion follows by Proposition \ref{appending x right}. Suppose $B$ is not $x$. Notice that the terms $x^4$ and $2x^5$ must arise in the expansion of $e(C_1)^2$ because $3\,\orde(C_2) \ge 6$ as $C_2 \ne x$. Since $x^4 = x^2 \cdot x^2$, we must have $e(C_1) = e(f(x,C_3)) = x^2 + e(C_3)^3$ where $C_3$ is another $f$-expression. Considering $(x^2 + e(C_3)^3)^2$ and the fact that $x^5 = x^2 \cdot x^3$ show that the expansion of $e(C_3)^3$ contains the term $x^3$, so $e(C_3)^3 = x^3$ and thus $C_3 = x$. Since $e(C_1) = x^2 + x^3$, it follows again by cancellation that $e(C_2)^3 = e(B)^3$ and so $e(C_2) = e(B)$. Since $f(x,x)$ is $e$-isolated, we have $C_1 = f(x,x)$.
\end{proof}

As can be seen, the above three propositions were established with an argument that works ``outside-in'' in the sense that it depends only on the $x$ and $f(x,x)$ that are being appended to $A$, $B$ by $f$, while $A$ and $B$ can be completely arbitrary. This argument is difficult to apply for $f$-expressions such as $f(C,B)$, where $B$ is arbitrary and $C$ is an $f$-expression such that $\dege(C) > \dege(f(x,x))$. Below we will introduce an argument that works, in some sense, ``inside-out.''

Let $f(A,B)$ be an $f$-expression. We will be examining the leading terms of the polynomial $e(f(A,B))$ by looking at the subexpressions from which they arise. For instance, $e(f(x,f(x,x))) = x^2 + e(f(x,x))^3 = x^2 + (x^2 + x^3)^3$ and we see that the term with the degree of the polynomial arises from the $f(x,x)$ by the product $(x^3)^3 = x^9$. The next lemma will show that the fact that this highest degree term arises from a subexpression $f(x,x)$ is very generally true.
\begin{lem}
\label{cores lemma}
Every summand contributing to the highest degree term of $e(f(A,B))$ must arise from an occurrence of $f(x,x)$ contained in $f(A,B)$, on expanding $e(f(A,B))$ in powers of $x$. 
\end{lem}
\begin{proof}
Consider the polynomial expansion of $e(f(A,B))$. Suppose $\dege(f(A,B)) = 2^m3^n$ for some $m, n \in \mathbb{N} \cup \{0\}$. Then the highest degree term of $e(f(A,B))$ is $px^{2^m3^n}$ for some $p \in \mathbb{N}$. Let $x_d$ denote an occurrence of $x$ in $f(A,B)$ such that at least one of the $p$ copies of $x^{2^m3^n}$ (we will denote this copy by $[x^{2^m3^n}]_\alpha$) in the expansion of $e(f(A,B))$ contains at least one factor of this occurrence of $x$, i.e. $[x^{2^m3^n}]_\alpha = x_d^2 \cdot x^{2^m3^n - 2}$ or $[x^{2^m3^n}]_\alpha = x_d^3 \cdot x^{2^m3^n - 3}$. Suppose that $x_d$ is contained in a subexpression $f(x_d,C)$ or $f(D,x_d)$ where $C \neq x$ and $D \neq x$. Considering the product $e(C)^3 \cdot x^{2^m3^n - 2}$ that arises from $f(x_d,C)$ and the product $e(D)^2 \cdot x^{2^m3^n - 3}$ that arises from $f(D,x_d)$, we see that in the expansion of $e(f(A,B))$ any occurrence of $x$ contained in $C$ or in $D$ will lead to a term with a power higher than $2^m3^n$. This is a contradiction, so $x_d$ must be contained in an $f(x,x_d)$ in $f(A,B)$.
\end{proof}
We will call an occurrence of $f(x,x)$ in the $f$-expression $f(A,B)$ a \textbf{core} of $f(A,B)$ if this occurrence of $f(x,x)$ gives rise to a summand contributing to the highest degree term of $e(f(A,B))$. Whenever an occurrence $f(x,x)'$ of $f(x,x)$ in $f(A,B)$ is a core of $f(A,B)$, it is clear that $e(f(x,x)')^{\frac{1}{3}\dege(f(A,B))} = (x^2 + x^3)^{\frac{1}{3}\dege(f(A,B))}$ must be a term of $e(f(A,B))$.

We define inductively what it means to \textbf{develop} an $f$-expression about a core:
\begin{enumerate}
\item
Start with $f(x,x)$ and label it a core of the $f$-expression to be developed. Then $f(x,x)$ is the $f$-expression at the first stage of the development.
\item
Let $A$ be the $f$-expression at the $n$th stage of the development where $n \ge 1$. Then the $f$-expression at the $(n + 1)$st stage of the development is either $f(A,C)$ where $C$ is an $f$-expression such that $3\,\dege(C) \le 2\,\dege(A)$ or $f(C,A)$ where $C$ is an $f$-expression such that $2\,\dege(C) \le 3\,\dege(A)$.
\end{enumerate}

Any $f$-expression can be developed inductively in the above manner, though the development may not be unique. For example, we can develop $f(x,f(x,x))$ only by the sequence of steps $f(x,x)$, $f(x,f(x,x))$ while we can develop $f(f(x,f(x,x)),f(f(x,x),x))$ by the sequence $f(x,x)$, $f(x,f(x,x))$, $f(f(x,f(x,x)),f(f(x,x),x))$ or by the sequence $f(x,x)$, $f(f(x,x),x)$, $f(f(x,f(x,x)),f(f(x,x),x))$. However, every development of an $f$-expression whose induced polynomial has degree $2^m3^n$ must consist of $m + n$ stages. 

Suppose that $C_1$ and $C_2$ are distinct cores of the $f$-expression $A$. Then we can find a subexpression $f(D_1,D_2)$ of $A$ such that either $D_1$ contains $C_1$ and $D_2$ contains $C_2$ or vice versa. Since $C_1$, $C_2$ are both cores, we must have $D_1 \ne D_2$ and $2\,\dege(D_1) = 3\,\dege(D_2)$. Since $f(D_1,D_2)$ is the $f$-expression at the $n$th stage of the development of $A$ about $C_1$ and the $f$-expression at the $n$th stage of the development of $A$ about $C_2$ for some $n \in \mathbb{N}$, we see that the development of $A$ about $C_1$ and the development of $A$ about $C_2$ differ at the $(n - 1)$st stage, i.e. these two developments are distinct. This analysis shows that an $f$-expression has a unique core whenever it has a unique development. Therefore, an $f$-expression has a unique development if and only if this $f$-expression has a unique core. Note that a ``development'' is defined as a property of an $f$-expression, not of a polynomial. So far as we know, for a given $f$-expression, the properties of having a unique core and of being $e$-isolated are independent of one another.

It is easy to see that an $f$-expression corresponding to a non-monic polynomial does not have a unique core. However, the number of cores of an $f$-expression do not necessarily equal the leading coefficient of the induced polynomial; consider $f(x,f(f(x,f(x,x)),f(f(x,x),x)))$, which has two cores while the polynomial it induces has a leading coefficient of 8. Moreover, it is easy to prove by induction the following

\begin{lem}
\label{non-monic condition}
Suppose $A \in \term(f;x)$ and $e(A)$ is non-monic. Then we have $\dege(A) = 2^p3^q$ where $p \ge 1$ and $q \ge 2$. 
\end{lem}
\begin{proof}
Let $A = f(B,C)$. If $2\,\dege(B) = 3\,\dege(C)$, then neither $B$ nor $C$ is $x$, so the value of $\dege(f(B,C)) = 2\,\dege(B) = 3\,\dege(C)$ will be divisible by $2$ and $3^2$, the latter because $\dege(C)$ is divisible by 3. If $2\,\dege(B) \ne 3\,\dege(C)$, then whichever of $B$ or $C$ contributes the higher degree term must be non-monic, and in that case we may assume inductively that either $\dege(B)$ or $\dege(C)$ satisfies the conclusion of the lemma.
\end{proof}

We shall call an $f$-expression $f(A,B)$ \textbf{disjoint} if $2\,\dege(A) < 3\,\orde(B)$ or if $3\,\dege(B) < 2\,\orde(A)$. $f(A,B)$ is called \textbf{hereditarily disjoint} if it is disjoint at every stage of its development about some core. We also consider $x$ to be (vacuously) hereditarily disjoint.

\begin{prop}
\label{hereditarily disjoint classification}
An $f$-expression $A$ is hereditarily disjoint if and only if it is either 
\begin{enumerate}
	\item $x$
	\item $f(x,U)$ for $U$ a hereditarily disjoint $f$-expression
	\item $f(U,x)$ for $U$ a hereditarily disjoint $f$-expression with $\orde(U) \ge 2$
	\item $f(f(x,x),U)$ for $U$ a hereditarily disjoint $f$-expression with $\orde(U) \ge 3$
\end{enumerate}
In these four cases, $\orde(A) = 1$, $2$, $3$, $4$ respectively, and it is easy to see that $A$ has a unique core in all cases except $A=x$.
\end{prop}
\begin{proof}
If $A$ belongs to one of the above four cases, then $A$ is hereditarily disjoint by definition. Now assume that $A$ is hereditarily disjoint, we will show that $A$ belongs to one of the above four cases. Suppose that $A \ne x$ and that for each core there are $n$ stages in the development of $A$ about that core. Suppose that for all $i \le n - 1$ the $f$-expression at the $i$th stage of the development of $A$ about each of its cores belongs to one of the above four cases. Then either $A = f(B,C)$ where $3\,\orde(C) > 2\,\dege(B)$ or $A = f(C,B)$ where $2\,\orde(C) > 3\,\dege(B)$. By the inductive hypothesis, we have $\orde(C) \le 4$, which forces $\dege(B) < 6$ for the case $A = f(B,C)$ and $\dege(B) < 3$ for the case $A = f(C,B)$. Thus, $A = f(B,C)$ implies $B = f(x,x)$ or $B = x$, and $A = f(C,B)$ implies $B = x$. This completes the induction.
\end{proof}

\begin{notn}
\label{ellipses form}
Let $A, B \in \term(f;x)$. Whenever we denote $A$ by $f(\ldots B \ldots)$, we mean that $B$ is a subexpression of $A$ and $B$ contains a core of $A$.
\end{notn}

Suppose $A := f( \ldots f(C,B) \ldots )$ is an $f$-expression of degree $2^p3^q$ where $B$ contains a core of $A$. Suppose $3\,\dege(B) = 2^m3^n$ and $2\,\dege(C) = 2^i3^j$. Define the \textbf{degree-gap between $C$ and $B$} to be the positive integer \begin{equation}\label{deg gap} \dgap(C,B) := 3\,\dege(B) - 2\,\dege(C) = 2^m3^n - 2^i3^j. \end{equation} In the $\dgap(-,-)$ notation we use, we will ignore the order of $B$ and $C$. In other words, we could also have written (\ref{deg gap}) as ``$\dgap(B,C) := 3\,\dege(B) - 2\,\dege(C) = 2^m3^n - 2^i3^j$'' (as we have already specified that $B$ contains a core of $A$). Now consider the expansion of $(e(C)^2 + e(B)^3)^{2^{p-m}3^{q-n}}$ and notice that the \textbf{highest degree monomial which can contain a factor coming from $C$ in the expansion of $e(A)$} is \begin{equation}\label{max term} \maxt_{A}(C) = x^{2^i3^j} \cdot (x^{2^m3^n})^{2^{p-m}3^{q-n}-1} = x^{2^i3^j + (2^m3^n)(2^{p-m}3^{q-n} - 1)}; \end{equation} here we are ignoring the coefficient of $x^{2^i3^j + (2^m3^n)(2^{p-m}3^{q-n} - 1)}$, as it is irrelevant at this point. The definition in (\ref{max term}) is relative to $A$ given, but we will abbreviate $\maxt_{A}(C)$ to $\maxt(C)$ where there is no danger of confusion. Of course, $B$ gives rise to the highest degree monomial $x^{2^p3^q} = x^{\dege(A)}$ of $A$. Notice that \begin{align*} \dege(A) - \deg(\maxt(C)) &= \deg((e(B)^3)^{2^{p-m}3^{q-n}}) - \deg(\maxt(C)) \\ &= 2^p3^q - (2^i3^j +(2^m3^n)(2^{p-m}3^{q-n} - 1)) \\ &= 2^m3^n - 2^i3^j \\ &= \dgap(C,B), \end{align*} so the degree-gap between $C$ and $B$ is \textbf{preserved in the expansion of $e(A)$}. This (and its analogue in the next paragraph) will be an important fact in Lemma \ref{gleftx} and Proposition \ref{e-isolated left appendage}, where we will prove that an $f$-expression is $e$-isolated by considering all possible developments that lead to an $f$-expression $e$-equivalent to the given one.

Similarly, in the opposite case, where $B := f( \ldots f(A,D) \ldots )$ is an $f$-expression of degree $2^p3^q$, $A$ contains a core of $B$, $2\,\dege(A) = 2^m3^n$, and $3\,\dege(D) = 2^i3^j$, we can define the \textbf{degree-gap between $A$ and $D$} to be the positive integer \begin{equation}\label{deg gap rev} \dgap(A,D) := 2\,\dege(A) - 3\,\dege(D) \end{equation} and notice that the \textbf{highest degree monomial which can contain a factor coming from $D$ in the expansion of $e(B)$} is \begin{equation}\label{max term rev} \maxt_{B}(D) = x^{2^i3^j + (2^m3^n)(2^{p-m}3^{q-n} - 1)}. \end{equation} Again, we will ignore the order of $A$ and $D$ in the $\dgap(-,-)$ notation we use, and we will abbreviate $\maxt_{B}(D)$ to $\maxt(D)$ where there is no danger of confusion. As before, we can observe that \[\dege(B) - \deg(\maxt(D)) = \dgap(A,D).\]

For an $f$-expression $A := f( \ldots B \ldots )$ where $\deg(e(B)) = 2^m3^{n-1}$, we say that $B$ is \textbf{$e$-isolated with respect to} $A$ if, for every development of every $f$-expression $e$-equivalent to $A$, we obtain the $f$-expression $B$ (not merely some $f$-expression $e$-equivalent to $B$) at the $(m + (n - 1))$st stage of the development. Note that if $B$ is $e$-isolated with respect to $A$, then $B$ must be $e$-isolated. The converse is not true, because even though $f(x,f(x,x))$ is $e$-isolated, it is not $e$-isolated with respect to $f(f(x,f(x,x)),f(f(x,x),x))$. Thus, $B$ being $e$-isolated with respect to $A$ is a stronger statement than $B$ being $e$-isolated. Note also that Lemma \ref{cores lemma} is equivalent to the statement that $f(x,x)$ is $e$-isolated with respect to every $f$-expression other than $x$.

\begin{mydef}
\label{agreement}
Let $A \in \term(f;x)$ and let $D_1(A)$, $D_2(A)$ denote two developments of $A$, not necessarily distinct. We shall say that $D_1(A)$, $D_2(A)$ \textbf{agree} at the $n$th stage if there exists $\hat{A} \in \term(f;x)$ such that $\hat{A}$ is the $f$-expression at the $n$th stage of both $D_1(A)$ and $D_2(A)$.
\end{mydef}

\begin{notn}
\label{stage of development}
Given $A \in \term(f;x)$, we will write $A^{[n]}$ for the $f$-expression at the $n$th stage of the development of $A$, provided that all developments of $A$ agree at the $n$th stage. Note that this notation refers only to developments of $A$, and not to developments of $f$-expressions $e$-equivalent to $A$, in contrast to the definition of ``$e$-isolated with respect to $A$'' in the paragraph preceding Definition \ref{agreement}.
\end{notn}

\begin{mydef}
\label{lexicographic ordering}
Let $p(x) = a_nx^n + a_{n-1}x^{n-1} +  \ldots  + a_1x + a_0$ and $q(x) = b_nx^n + b_{n-1}x^{n-1} +  \ldots  + b_1x + b_0$ be two polynomials with nonnegative coefficients. Suppose $p(x) \ne q(x)$, and let $m$ be the greatest integer such that $a_m \ne b_m$. Then we say that $p(x)$ is \textbf{lexicographically greater than} $q(x)$, denoted by $p(x) >_L q(x)$, if $a_m > b_m$.
\end{mydef}

\begin{lem}
\label{gleftx}
Let $A = f(\ldots f(x',B) \ldots)$ be an $f$-expression where $x' := x$ for the purpose of distinguishing it from the other occurrences of the variable $x$ in $A$, $\dege(A) = 2^p3^q$, and $B$ contains a core of $A$. Suppose that for every $f$-expression $C$ in the ellipses $( \ldots )$ of $A$ we have $\deg(\maxt(C)) \le \deg(\maxt(x'))$. Suppose there exists either an $f$-expression $\bar{A} = f(\ldots f(U,B) \ldots)$ such that $\dege(\bar{A}) = 2^p3^q$, $U \ne x$, and $B$ contains a core of $\bar{A}$ or an $f$-expression $\hat{A} = f(\ldots f(B,V) \ldots)$ such that $\dege(\hat{A}) = 2^p3^q$ and $B$ contains a core of $\hat{A}$. Then $e(A) <_L e(\bar{A})$ if $\bar{A}$ exists and $e(A) <_L e(\hat{A})$ if $\hat{A}$ exists.
\end{lem}
\begin{proof}
By our assumption, the subexpressions in the ellipses of $A$ give rise to terms with powers no higher than that of $\maxt(x')$. Suppose $e(B)^3$ is of degree $2^m3^n$. Notice that $(e(B)^3)^{2^{p-m}3^{q-n}} = (e(B)^2)^{2^{p-m-1}3^{q-n+1}}$ is common to both $e(A)$ and $e(\bar{A})$ (if $\bar{A}$ exists), and is common to both $e(A)$ and $e(\hat{A})$ (if $\hat{A}$ exists). We have $\deg(e(A) - (e(B)^3)^{2^{p-m}3^{q-n}}) = \deg(\maxt(x'))$, $\deg(e(\bar{A}) - (e(B)^3)^{2^{p-m}3^{q-n}}) \ge \deg(\maxt(U))$, and $\deg(e(\hat{A}) - (e(B)^2)^{2^{p-m-1}3^{q-n+1}}) \ge \deg(\maxt(V))$. Since $U \ne x$, we have $\deg(\maxt(U)) > \deg(\maxt(x'))$, so $\deg(e(A) - (e(B)^3)^{2^{p-m}3^{q-n}}) < \deg(e(\bar{A}) - (e(B)^3)^{2^{p-m}3^{q-n}})$. It follows that $e(A) <_L e(\bar{A})$ (in the case that $\bar{A}$ exists) as desired. Since $\dgap(x',B) = 2^m3^n - 2 > 2^{m+1}3^{n-1} - 3 \ge 2^{m+1}3^{n-1} - 3\,\dege(V) = \dgap(B,V)$, we have $\deg(\maxt(x')) = 2^p3^q - \dgap(x',B) < 2^p3^q - \dgap(B,V) = \deg(\maxt(V))$, so $\deg(e(A) - (e(B)^3)^{2^{p-m}3^{q-n}}) < \deg(e(\hat{A}) - (e(B)^2)^{2^{p-m-1}3^{q-n+1}})$. It follows that $e(A) <_L e(\hat{A})$ (in the case that $\hat{A}$ exists) as desired.
\end{proof}

\begin{lem}
\label{grightx}
Let $A = f(\ldots f(B,x') \ldots)$ be an $f$-expression such that $x' := x$, $\dege(A) = 2^p3^q$, and $B$ contains a core of $A$. Suppose $\bar{A} = f(\ldots f(B,U) \ldots)$ is an $f$-expression such that $U \ne x$, $\dege(\bar{A}) = 2^p3^q$, and $B$ contains a core of $\bar{A}$. Suppose that for every $f$-expression $C$ in the ellipses $( \ldots )$ of $A$ we have $\deg(\maxt(C)) \le \deg(\maxt(x'))$. Then $e(A) <_L e(\bar{A})$.
\end{lem}
\begin{proof}
By our assumption, the subexpressions in the ellipses of $A$ give rise to terms with powers no higher than that of $\maxt(x')$. Suppose $e(B)^2$ is of degree $2^m3^n$. Notice that $(e(B)^2)^{2^{p-m}3^{q-n}}$ is common to both $e(A)$ and $e(\bar{A})$. We have $\deg(e(A) - (e(B)^2)^{2^{p-m}3^{q-n}}) = \deg(\maxt(x'))$ and $\deg(e(\bar{A}) - (e(B)^2)^{2^{p-m}3^{q-n}}) \ge \deg(\maxt(U))$. Since $U \ne x$, we have $\deg(\maxt(U)) > \deg(\maxt(x'))$, so $\deg(e(\bar{A}) - (e(B)^2)^{2^{p-m}3^{q-n}}) > \deg(e(A) - (e(B)^2)^{2^{p-m}3^{q-n}})$. It follows that $e(A) <_L e(\bar{A})$ as desired.
\end{proof}

\begin{lem}
\label{maxt left x}
Let $U \in \term(f;x)$ where $U \ne x$. Let $B_U^{(1)} = f(x_1,U)$ where $x_1 := x$. For every positive integer $n$, let $B_U^{(n + 1)} = f(x_{n+1}, B_U^{(n)})$, where $x_{n+1} := x$. Let $A = f(\ldots B_U^{(n)} \ldots)$ where $n \in \mathbb{N}$. Then $\deg(\maxt(x_1)) > \deg(\maxt(x_2)) > \ldots > \deg(\maxt(x_n))$.
\end{lem}
\begin{proof}
Since $B_U^{(n)}$ contains a core of $A$, $U$ must contain a core of $A$. Notice that $\dgap(x_1,U) < \dgap(x_2,B_U^{(1)}) < \ldots < \dgap(x_n,B_U^{(n-1)})$. Since $\deg(\maxt(x_1)) = \dege(A) - \dgap(x_1,U)$, $\deg(\maxt(x_2)) = \deg(A) - \dgap(x_2,B_U^{(1)})$, \ldots, and $\deg(\maxt(x_n)) = \dege(A) - \dgap(x_n,B_U^{(n-1)})$, the conclusion immediately follows.
\end{proof}

Analogously, we have the following
\begin{lem}
\label{maxt right x}
Let $V \in \term(f;x)$ where $V \ne x$. Let $C_V^{(1)} = f(V,x_1)$ where $x_1 := x$. For every positive integer $n$, let $C_V^{(n+1)} = f(C_V^{(n)},x_{n+1})$, where $x_{n+1} := x$. Let $A = f(\ldots C_V^{(n)} \ldots)$ where $n \in \mathbb{N}$. Then $\deg(\maxt(x_1)) > \deg(\maxt(x_2)) > \ldots > \deg(\maxt(x_n))$.
\end{lem}

\begin{lem}
\label{maxt mixed x}
Let $A = f(\ldots f(f(x_1,B),x_2) \ldots)$ be an $f$-expression such that $x_1 := x$, $x_2 := x$ and $B \ne x$. Then $\deg(\maxt(x_1)) > \deg(\maxt(x_2))$.
\end{lem}
\begin{proof}
Notice that $B$ must contain a core of $A$. We have $\dgap(x_1,B) = 3\,\dege(B) - 2 < 6\,\dege(B) - 3 = \dgap(f(x_1,B),x_2)$. Since $\deg(\maxt(x_1)) = \dege(A) - \dgap(x_1,B)$ and $\deg(\maxt(x_2)) = \dege(A) - \dgap(f(x_1,B),x_2)$, the conclusion immediately follows.
\end{proof}

\begin{prop}
\label{e-isolated left appendage}
Let $A = f( \ldots f(x',B) \ldots )$ be an $f$-expression such that $x' := x$. Suppose that for every $f$-expression $C$ in the ellipses $( \ldots )$ of $A$ we have $\deg(\maxt(C)) \le \deg(\maxt(x'))$. Suppose that $B$ is $e$-isolated with respect to $A$. Then $f(x',B)$ is $e$-isolated with respect to $A$.
\end{prop}
\begin{proof}
Suppose $3\,\dege(B) = 2^m3^n$. Let $A'$ be an $f$-expression $e$-equivalent to $A$. Then $A'^{[m+(n-1)]} = B$ by our assumption, so the $(m+n)$th stage of every development of $A'$ is either $f(U,B)$ or $f(B,V)$ for some $U, V \in \term(f;x)$. Suppose that either $A' = f(\ldots f(U,B) \ldots)$ where $U \ne x$ or $A' = f(\ldots f(B,V) \ldots)$. By Lemma \ref{gleftx} we have $e(A) <_L e(A')$, which is a contradiction. It follows that the $(m+n)$th stage of every development of $A'$ must be of the form $f(U,B)$ where $U = x$. Hence $A'^{[m+n]} = f(x,B)$ as desired.
\end{proof}

\begin{notn}
\label{x left appendage def}
Let $B^{(1)} = f(x,x)$. For every positive integer $n$, let $B^{(n + 1)} = f(x, B^{(n)})$. Also, we let $B^{(0)} = x$.
\end{notn}

Notation \ref{x left appendage def} is to be used for the remainder of this paper, and is not to be confused with the \textit{ad hoc} notations set up in Lemma \ref{maxt left x} and Lemma \ref{maxt right x}.

\begin{cor}
\label{x left appendage separation}
Let $A = f(\ldots f(x',B^{(m)}) \ldots)$ be an $f$-expression where $x' := x$. Suppose that for every $f$-expression $C$ in the ellipses $( \ldots )$ of $A$ we have $\deg(\maxt(C)) \le \deg(\maxt(x'))$. Then $f(x',B^{(m)})$ is $e$-isolated with respect to $A$.
\end{cor}
\begin{proof}
We know that $B^{(1)} = f(x,x)$ is $e$-isolated with respect to $A$. Suppose we know that $B^{(k)}$ is $e$-isolated with respect to $A$ for some $1 \le k \le m$. Writing $A$ as $f(\ldots f(x'',B^{(k)}) \ldots)$ where $x'' := x$, we see by Lemma \ref{maxt left x} that for every $f$-expression $D$ in the ellipses of $f(\ldots f(x'',B^{(k)}) \ldots)$ we have $\deg(\maxt(D)) \le \deg(\maxt(x''))$. It then follows by Proposition \ref{e-isolated left appendage} that $f(x'',B^{(k)})$ is $e$-isolated with respect to $A$. This completes the induction.
\end {proof}

\begin{lem}
\label{left app lex min}
Let $A = f(\ldots f(x',B^{(m)}) \ldots)$ be an $f$-expression where $x' := x$. Let $\bar{A} = f(\ldots C \ldots)$ be an $f$-expression such that $\dege(\bar{A}) = \dege(A)$, $C \ne f(x',B^{(m)})$, and $\dege(C) = \dege(f(x',B^{(m)}))$. Suppose that for every $f$-expression $D$ in the ellipses of $A$ we have $\deg(\maxt(D)) \le \deg(\maxt(x'))$. Then $e(A) <_L e(\bar{A})$.
\end{lem}
\begin{proof}
$C$ has a unique core by Lemma \ref{non-monic condition}. Since $C \ne f(x',B^{(m)})$ and $\dege(C) = \dege(f(x',B^{(m)}))$, we can find some $k \le m$ in $\mathbb{N}$ such that $C^{[k+1]} = f(U,B^{(k)})$ where $U \ne x$. Then we can write $\bar{A}$ as $f(\ldots f(U,B^{(k)}) \ldots)$. Since $f(x',B^{(m)})^{[k+1]} = f(x'',B^{(k)})$ where $x'' := x$, we can write $A$ as $f(\ldots f(x'',B^{(k)}) \ldots)$. We see by Lemma \ref{maxt left x} that for every $f$-expression $E$ in the ellipses of $f(\ldots f(x'',B^{(k)}) \ldots)$ we have $\deg(\maxt(E)) \le \deg(\maxt(x''))$. It follows by Lemma \ref{gleftx} that $e(A) <_L e(\bar{A})$.
\end{proof}

Lexicographic ordering on polynomials with nonnegative integer coefficients is a well-ordering. In particular, for all $m$ and $n$ the set of all polynomials of degree $2^m3^n$ induced by $f$-expressions contains exactly one lexicographically minimal polynomial, and we will see that this polynomial corresponds to an $e$-isolated $f$-expression.

As an illustration, we claim that the $f$-expression $f(f(f(x,f(x,x)),x),x)$ leads to the lexicographically minimal polynomial with degree $36 = (2^2)(3^2)$. The following lemma gives the general rule.
\begin{lem}
\label{lexico min f-expression}
For any $A \in \term(f;x)$, let $u(A) = f(x,A)$ and $v(A) = f(A,x)$. The $f$-expression that induces the lexicographically minimal polynomial of degree $2^m3^n$ with $n \ge 1$ is $v(\ldots v(u(\ldots u(x) \ldots)) \ldots)$ with $m$ $v$'s followed by $n$ $u$'s in left-to-right order. Moreover, this $f$-expression is $e$-isolated.
\end{lem}
\begin{proof}
Let $A := v(\ldots v(u(\ldots u(x) \ldots)) \ldots)$ with $m$ $v$'s followed by $n$ $u$'s in left-to-right order, and let $A' \in \term(f;x)$ be such that $A' \ne A$ and $\dege(A') = \dege(A)$. It is clear that $A$ has exactly one core. Let $k$ be the largest positive integer such that the development of $A$ agrees with every development of $A'$ at the $k$th stage. Let $C$ denote the $f$-expression at the $k$th stage of the development of $A$. Suppose $k < n$. Then we can write $A$ as $f(\ldots f(x_1,C) \ldots)$ and we can write $A'$ as either $f(\ldots f(U,C) \ldots)$ or $f(\ldots f(C,V) \ldots)$ where $x_1 := x$ and $U, V \in \term(f;x)$ such that $U \ne x$. By Lemma \ref{maxt left x}, Lemma \ref{maxt right x}, and Lemma \ref{maxt mixed x}, we have $\deg(\maxt(x_1)) > \deg(\maxt(D))$ for every $f$-expression $D$ in the ellipses of $f(\ldots f(x_1,C) \ldots)$. It follows by Lemma \ref{gleftx} that $e(A) <_L e(A')$. Suppose $k \ge n$. Then we can write $A$ as $f(\ldots f(C,x_2) \ldots)$ and we can write $A'$ as $f(\ldots f(C,W) \ldots)$, where $x_2 := x$ and $W \in \term(f;x)$ such that $W \ne x$. By Lemma \ref{maxt right x}, we have $\deg(\maxt(x_2)) > \deg(\maxt(E))$ for every $f$-expression $E$ in the ellipses of $f(\ldots f(C,x_2) \ldots)$. It follows by Lemma \ref{grightx} that $e(A) <_L e(A')$. Thus, in all cases we have $e(A) <_L e(A')$. It immediately follows from this analysis that $A$ is $e$-isolated, though we could well have proven this particular fact by repeatedly applying Proposition \ref{appending x left} and Proposition \ref{appending x right}.
\end{proof}

As we will see, this concept of lexicographic minimality can be applied to prove that many classes of $f$-expressions are $e$-isolated. This concept also illustrates one advantage of working with the single-variable case of Friedman's problem, because it is less clear how one would lexicographically order multiple-variable polynomials.

\begin{prop}
\label{lexico min appendage}
Let $f(f(F,x'),B)$ be an $f$-expression such that $e(f(f(F,x'),B))$ has degree $2^m3^n$, where $x':= x$. Suppose that
\begin{enumerate}
	\item $3\,\dege(B) \le \deg(\maxt(x'))$, where we note that $\deg(\maxt(x')) = 3 + 2\,\dege(F)$
	\item $e(F)$ is lexicographically minimal among polynomials of degree $2^{m-2}3^n$ induced by $f$-expressions as characterized in Lemma \ref{lexico min f-expression}
	\item $a$, $b$ are $f$-expressions such that $e(f(a,b)) = e(f(f(F,x'),B))$.
\end{enumerate}
Then we must have $a = f(F,x')$ and $e(b) = e(B)$.
\end{prop}
\begin{proof}
First notice that $m \ge 2$ and that, by Lemma \ref{lexico min f-expression}, $e(f(F,x'))$ is lexicographically minimal among polynomials of degree $2^{m-1}3^n$ induced by $f$-expressions. Writing $f(f(F,x'),B)$ as $f(\ldots f(x'',B^{(n-1)}) \ldots)$ where $x'' := x$, we see by Lemma \ref{maxt right x} and Lemma \ref{maxt mixed x} that $\deg(\maxt(C)) \le \deg(\maxt(x''))$ for every $f$-expression $C$ in the ellipses of $f(\ldots f(x'',B^{(n-1)}) \ldots)$. It follows by Corollary \ref{x left appendage separation} that $f(a,b)^{[n]} = B^{(n)} = f(f(F,x'),B)^{[n]}$. Assume we know that $f(a,b)^{[k]} = f(f(F,x'),B)^{[k]}$ for some positive integer $n \le k < m - 1 + n$. Suppose $f(a,b) = f(\ldots f(f(f(F,x'),B)^{[k]},D) \ldots)$ where $D \ne x$. Writing $f(f(F,x'),B)$ as $f(\ldots f(f(f(F,x'),B)^{[k]},x''') \ldots)$ where $x''' := x$, we see by Lemma \ref{maxt right x} that $\deg(\maxt(E)) \le \deg(\maxt(x'''))$ for every $f$-expression $E$ in the ellipses of $f(\ldots f(f(f(F,x'),B)^{[k]},x''') \ldots)$. It follows by Lemma \ref{grightx} that $e(f(a,b)) >_L e(f(f(F,x'),B))$, which is a contradiction. Thus, we must have $f(a,b)^{[k+1]} = f(f(f(F,x'),B)^{[k]},x) = f(f(F,x'),B)^{[k+1]}$. This completes the induction. Consequently, we have $f(a,b)^{[(m-1)+n]} = f(F,x)$, so it follows that $a = f(F,x)$. Finally, $e(f(f(F,x),b)) = e(f(f(F,x),B))$ implies that $e(b) = e(B)$.
\end{proof}

Next, we will apply the concept of lexicographic minimality to prove another general result. First, let $A$, $B$ be fixed $f$-expressions. We examine some of the restrictions $e(f(A,B)) = e(f(C,D))$ imposes on the $f$-expressions $C$ and $D$. To avoid triviality, assume $e(B) \ne e(D)$. We have $e(A)^2 + e(B)^3 = e(C)^2 + e(D)^3$ iff \begin{equation}\label{rearranging terms} e(B)^3 - e(D)^3 = e(C)^2 - e(A)^2 \end{equation} iff $(e(B) - e(D))(e(B)^2 + e(B)e(D) + e(D)^2) = e(C)^2 - e(A)^2$. Since $e(B) - e(D) \ne 0$, we know that $e(B) - e(D)$ is not a constant. So we have $\deg(e(C)^2 - e(A)^2) > \max(2\,\dege(B),2\,\dege(D))$. In particular, if $\dege(A) \le \dege(B)$, then we must have \begin{equation}\label{newlabel8} 2\,\dege(C) > \max(2\,\dege(B),2\,\dege(D)). \end{equation} This implies that $2\,\dege(C) > 2\,\dege(B) \ge 2\,\dege(A)$, which implies that $\dege(C) > \dege(A)$. It then follows by (\ref{rearranging terms}) that $e(B)^3 >_L e(D)^3$, so we have $e(B) >_L e(D)$. Of course, from (\ref{newlabel8}) we also obtain $\dege(C) > \dege(D)$. Therefore, we have established the following
\begin{lem}
\label{e-equiv condition}
Let $A$, $B$ be fixed $f$-expressions such that $\dege(A) \le \dege(B)$. If $e(f(A,B)) = e(f(C,D))$ and $e(B) \ne e(D)$, then $\dege(C) > \dege(A)$, $e(B) >_L e(D)$, and $\dege(C) > \dege(D)$.
\end{lem}
\begin{notn}
\label{lead function}
In what follows, we will denote the coefficient of the highest degree term of a polynomial $p(x)$ by $\lead(p(x))$. For example, we have $\lead(e(f(f(x,f(x,x)),f(f(x,x),x)))) = 2$.
\end{notn}

For the next two propositions, fix $f(A,f(E,B^{(m)})) \in \term(f;x)$ where $\dege(E) \le \dege(B^{(m)})$ and $\dege(A) \le \dege(f(E,B^{(m)}))$, so $e(f(A,f(E,B^{(m)})))$ is monic, of degree $3^{m+2}$. Writing $f(A,f(E,B^{(m)}))$ as $f(A,f(E,f(x',B^{(m-1)})))$, we see that $2\,\dege(A) < \deg(\maxt(E)) < \deg(\maxt(x'))$, because $\dgap(A,f(E,B^{(m)})) = 3^{m+2} - 2\,\dege(A) \ge 3^{m+2} - 2\,\dege(f(E,B^{(m)})) = 3^{m+2} - 2 \cdot 3^{m+1} = 3^{m+1} > 3^{m+1} - 2\,\dege(E) = \dgap(E,B^{(m)})$ and $\dgap(E,B^{(m)}) = 3^{m+1} - 2\,\dege(E) \ge 3^{m+1} - 2\,\dege(B^{(m)}) = 3^{m+1} - 2 \cdot 3^m = 3^m > 3^m - 2 = \dgap(x',B^{(m-1)})$. It follows that $B^{(m)}$ is $e$-isolated with respect to $f(A,f(E,B^{(m)}))$ by Corollary \ref{x left appendage separation}. Suppose that \[e(f(A,f(E,B^{(m)}))) = e(f(C,D)).\] Then $D = f(F,B^{(m)})$ for some $F \in \term(f;x)$, so \begin{equation}\label{newlabel1} e(f(A,f(E,B^{(m)}))) = e(f(C,f(F,B^{(m)}))). \end{equation} In the following two propositions and their corollaries, we will use Lemma \ref{e-equiv condition} along with the concept of lexicographic minimality to show that, for a large class of $f$-expressions that $E$ may assume, the preceeding equality implies that \begin{equation}\label{partial determination} F = E \;\;\; \& \;\;\; e(C) = e(A).\end{equation}
\begin{prop}
\label{determination of dege}
Let $f(A,f(E,B^{(m)})) \in \term(f;x)$ have the property that $\dege(E) \le \dege(B^{(m)})$ and $\dege(A) \le \dege(f(E,B^{(m)}))$. Suppose that $e(f(A,f(E,B^{(m)}))) = e(f(C,D))$. Then $D = f(F,B^{(m)})$ for some $F \in \term(f;x)$. Furthermore, we must have $e(E) \ge_L e(F)$ and $\dege(E) = \dege(F)$.
\end{prop}
\begin{proof}
We have already proved the first part of the conclusion in the discussion leading up to (\ref{newlabel1}), so it remains to show that $e(E) \ge_L e(F)$ and $\dege(E) = \dege(F)$. If $e(F) = e(E)$, then we are done. Suppose $e(F) \ne e(E)$. Then $e(f(F,B^{(m)})) \ne e(f(E,B^{(m)}))$. By Lemma \ref{e-equiv condition} we must have $\dege(C) > \dege(A)$ and $e(f(E,B^{(m)})) >_L e(f(F,B^{(m)}))$, and it follows that \begin{equation}\label{lexico greater} e(E) >_L e(F).\end{equation} Therefore, if $\dege(E) \ne \dege(F)$, we must have $\dege(E) > \dege(F)$; so let us assume $\dege(E) > \dege(F)$ and obtain a contradiction. We have \[e(A)^2 + (e(E)^2 + e(B^{(m)})^3)^3 = e(C)^2 + (e(F)^2 + e(B^{(m)})^3)^3.\] After cancelling out the common $e(B^{(m)})^9$ from both sides, we see that the degree of the left-hand side is $\deg(e(E)^2e(B^{(m)})^6)$ and the highest-degree term on the right-hand side must be $e(C)^2$ because $\deg(e(F)^2e(B^{(m)})^6) < \deg(e(E)^2e(B^{(m)})^6)$. This means that \begin{equation}\label{deg equality} 2\,\dege(C) = \deg(e(E)^2e(B^{(m)})^6). \end{equation} Then the coefficient of the highest degree term of the left-hand side must be $\lead(3e(E)^2e(B^{(m)})^6) = 3\,\lead(e(E))^2\lead(e(B^{(m)}))^6 = 3\,\lead(e(E))^2$, and the coefficient of the highest degree term of the right-hand side must be $\lead(e(C)^2) = \lead(e(C))^2$. We must have $\lead(e(C))^2 = 3\,\lead(e(E))^2$, from which it follows that \begin{equation}\label{lead equality} \lead(e(C)) = \sqrt{3}\,\lead(e(E)), \end{equation} which is not even rational. This is a contradiction, so we must have $\dege(E) = \dege(F)$.
\end{proof}
\begin{cor}
\label{lexico min branch}
Under the hypotheses of Proposition \ref{determination of dege}, if $E$ induces the lexicographically minimal polynomial with degree $\dege(E)$, then $F = E$ and $e(C) = e(A)$.
\end{cor}
\begin{proof}
We must have $e(E) = e(F)$ or $e(E) <_L e(F)$. Since $e(E) \ge_L e(F)$ by Proposition \ref{determination of dege}, this forces $e(F) = e(E)$, from which the conclusion follows by Lemma \ref{lexico min f-expression}.
\end{proof}

The next proposition will be a generalization of Proposition \ref{determination of dege}. We will demonstrate in Corollary \ref{lexico min branch gen} that $(\ref{partial determination})$ follows even if $f(E,B^{(m)})$ assumes values in a more general class of $f$-expressions than that in Corollary \ref{lexico min branch}.

\begin{notn}
\label{tree branch notn}
For every integer $1 \le j \le m$ define the function $Y_j: \term(f;x)$ $\longrightarrow$ $\term(f;x)$ by $Y_j(A) = f(A,B^{(j)})$ for every $A \in \term(f;x)$; in other words, we can write $Y_j = f(-,B^{(j)})$ as a one-argument function.
\end{notn}
\begin{prop}
\label{determination of dege gen}
Assume Notation \ref{tree branch notn} above. Let $1 \le j \le m$ be a positive integer and $d_1$, \ldots, $d_{j-1}$, $d_j$ be a sequence of positive integers such that $m = d_1 > d_2 > \ldots > d_{j-1} > d_j \ge 1$. Let $U \in \term(f;x)$ with $\dege(U) \le \dege(B^{(d_j)}) = 3^{d_j}$, and let $E = Y_{d_2}(Y_{d_3}(\ldots Y_{d_j}(U) \ldots))$, so $f(E,B^{(m)}) = Y_{d_1}(E) = Y_{d_1}(Y_{d_2}(\ldots Y_{d_j}(U) \ldots)) = Y_{d_1} \circ Y_{d_2} \circ \ldots \circ Y_{d_j}(U)$. Let $A \in \term(f;x)$ be such that $\dege(A) \le \dege(f(E,B^{(m)}))$. Suppose that $e(f(C,D)) = e(f(A,f(E,B^{(m)})))$. Then there exists $V \in \term(f;x)$ such that $D = Y_{d_1} \circ Y_{d_2} \circ \ldots \circ Y_{d_j}(V)$, $e(U) \geq_L e(V)$, and $\dege(V) = \dege(U)$.
\end{prop}
\begin{proof}
We will prove this by induction on $j$. We have already proved the case $j = 1$ in Proposition \ref{determination of dege}. Now assume that the statement of this proposition holds for all $1 \le k \le j - 1$. In what follows, we will prove this proposition for $j$. 

Since $e(f(C,D)) = e(f(A,f(E,B^{(m)})))$ and $f(E,B^{(m)}) = Y_{d_1} \circ Y_{d_2} \circ \ldots \circ Y_{d_{j-1}}(f(U,B^{(d_j)}))$, it follows from our inductive hypothesis (with $f(U,B^{(d_j)})$ in the role of $U$) that $D = Y_{d_1} \circ Y_{d_2} \circ \ldots \circ Y_{d_{j-1}}(\tilde{V})$ for some $\tilde{V} \in \term(f;x)$ such that $e(f(U,B^{(d_j)})) \geq_L e(\tilde{V})$ and $\dege(\tilde{V}) = \dege(f(U,B^{(d_j)}))$. We claim that $\tilde{V} = f(V,B^{(d_j)})$ for some $V \in \term(f;x)$ such that $e(V) \le_L e(U)$. Notice that we must have $\tilde{V} = f(V,W)$ for some $V,W \in \term(f;x)$, where $\dege(W) = \dege(B^{(d_j)})$. Notice also that $f(U,B^{(d_j)}) = f(U,f(x',B^{(d_j-1)}))$ where $x' := x$, and $\deg(e(U)^2) = 2\,\dege(U) \le 2\,\dege(B^{(d_j)}) < 2 + 2\,\dege(B^{(d_j)}) = \deg(x'^2) + 2\,\deg(e(B^{(d_j-1)})^3) = \deg(\maxt(x'))$. Suppose $W \ne B^{(d_j)}$. Then by Lemma \ref{left app lex min} we have $e(f(U,B^{(d_j)})) <_L e(f(V,W)) = e(\tilde{V})$, which is a contradiction. Hence we must have $W = B^{(d_j)}$. Since $e(f(U,B^{(d_j)})) \geq_L e(\tilde{V}) = e(f(V,B^{(d_j)}))$, we have $e(V) \le_L e(U)$ as claimed.

Now we want to show that $\dege(V) = \dege(U)$. Suppose \begin{equation}\label{newlabel2} \dege(V) < \dege(U). \end{equation} We will show that this leads to a contradiction. We have $e(f(A,f(E,B^{(m)}))) = e(f(A,Y_{d_1} \circ Y_{d_2} \circ \ldots \circ Y_{d_j}(U))) = e(A)^2 + ((( \ldots ((e(U)^2 + e(B^{(d_j)})^3)^2 + e(B^{(d_{j-1})})^3)^2  \ldots )^2 + e(B^{(d_2)})^3)^2 + e(B^{(d_1)})^3)^3 = e(C)^2 + ((( \ldots ((e(V)^2 + e(B^{(d_j)})^3)^2 + e(B^{(d_{j-1})})^3)^2  \ldots )^2 + e(B^{(d_2)})^3)^2 + e(B^{(d_1)})^3)^3 = e(f(C,Y_{d_1} \circ Y_{d_2} \circ \ldots \circ Y_{d_j}(V))) = e(f(C,D))$. Notice that, in the polynomial expansion of the preceeding five equalities, the terms (excluding the $e(A)^2$ and $e(C)^2$) that do not contain $e(U)^2$ as a factor or $e(V)^2$ as a factor, $(e(B^{(d_2)})^3)^2(e(B^{(d_1)})^3)^2$ for example, are common to both sides of the third equality and can thus be subtracted off from these two sides. Subtracting off these common terms leaves \[\deg(e(U)^2e(B^{(d_j)})^3e(B^{(d_{j-1})})^3 \ldots e(B^{(d_2)})^3e(B^{(d_1)})^6)\] as the degree of the left-hand side of the third equality. Since \[\deg(e(V)^2e(B^{(d_j)})^3e(B^{(d_{j-1})})^3 \ldots e(B^{(d_2)})^3e(B^{(d_1)})^6)\] is less than the degree of the left-hand side of the third equality by (\ref{newlabel2}), it follows that $2\,\dege(C)$ must equal the degree of the left-hand side of the third equality. We see from the third equality that the coefficient of the highest degree term of the left-hand side is $\lead(3 \cdot 2^{j-1}e(U)^2e(B^{(d_j)})^3e(B^{(d_{j-1})})^3 \ldots e(B^{(d_2)})^3e(B^{(d_1)})^6) = 3 \cdot 2^{j-1}\lead(e(U))^2\lead(e(B^{(d_1)}))^6 = 3 \cdot 2^{j-1}\lead(e(U))^2$ and that the coefficient of the highest degree term of the right-hand side is $\lead(e(C)^2) = \lead(e(C))^2$. It follows that we must have $\lead(e(C))^2 = 3 \cdot 2^{j-1}\lead(e(U))^2$, which implies that $\lead(e(C)) = \sqrt{3 \cdot 2^{j-1}}\,\lead(e(U))$, which is not even rational. This is a contradiction, so we must have $\dege(U) = \dege(V)$. This completes the induction.
\end{proof}
\begin{cor}
\label{lexico min branch gen}
Under the hypotheses of Proposition \ref{determination of dege gen}, if $e(U)$ is lexicographically minimal among polynomials with degree $\dege(U)$ induced by $f$-expressions, then $V = U$ and $e(A) = e(C)$.
\end{cor}
\begin{proof}
We must have $e(V) = e(U)$ or $e(V) >_L e(U)$. Since $e(U) \geq_L e(V)$ by Proposition \ref{determination of dege gen}, we must have $e(V) = e(U)$, from which the conclusion follows by Lemma \ref{lexico min f-expression}.
\end{proof}

In the proof of Proposition \ref{determination of dege} (and analogously Proposition \ref{determination of dege gen}), we saw that, under the hypotheses of this proposition and given $\dege(E) > \dege(F)$, the subexpression $C$ is not able to ``make up'' for the difference between $e(f(E,B^{(m)}))$ and $e(f(F,B^{(m)}))$, and hence $e(f(C,f(F,B^{(m)}))) \neq e(f(A,f(E,B^{(m)})))$. We will analyze and make use of this phenomenon extensively in what follows, where we consider the developments of $f$-expressions more complicated than $A$ of Proposition \ref{e-isolated left appendage}.

Consider $A, B_0, E_1 \in \term(f;x)$ where $B_0$, $E_1$ are subexpressions of $A$, $\dege(A) = 2^p3^q$, $\dege(B_0) = 2^m3^n$, and $\dege(E_1) = 2^i3^j$. For the remainder of this paper, assume the following three
\begin{assmpt}
\label{newlabel3}
$B_0$ contains all the cores of $A$ and is $e$-isolated with respect to $A$.
\end{assmpt}
\begin{assmpt}
\label{newlabel4}
$A^{[m + n + 1]}$ is either $f(E_1,B_0)$ or $f(B_0,E_1)$. In other words, we can write $A = f(\ldots f(E_1,B_0) \ldots)$ or $A = f(\ldots f(B_0,E_1) \ldots)$ respectively.
\end{assmpt}
\begin{assmpt}
\label{newlabel5}
For every occurrence $x'$ of $x$ in $E_1$ and every occurrence $x''$ of $x$ in the ellipses of the expressions for $A$ shown in Assumption \ref{newlabel4} we have $\deg(\maxt(x')) > \deg(\maxt(x''))$.
\end{assmpt}
\begin{rem}
\label{third assumption}
Note that Assumption \ref{newlabel5} is a more general version of the corresponding hypothesis in Proposition \ref{e-isolated left appendage}.
\end{rem}
In the restricted version of Friedman's problem we will study below, we will show that whether or not $A^{[m + n + 1]}$ is $e$-isolated with respect to $A$ is related to the solution sets of certain exponential Diophantine equations.

In either of the cases $A = f(\ldots f(E_1,B_0) \ldots)$ or $A = f(\ldots f(B_0,E_1) \ldots)$ we have \begin{equation}\label{dgap expression} \dgap(E_1,B_0) = 2^{m+\pi_1}3^{n+\pi_2} - 2^{i+\pi_2}3^{j+\pi_1} \end{equation} where $\{\pi_1,\pi_2\} = \{0,1\}$. Note that $\pi_1 = 1$ corresponds to the case $A = f(\ldots f(B_0,E_1) \ldots)$ and $\pi_2 = 1$ corresponds to the case $A = f(\ldots f(E_1,B_0) \ldots)$. We could attempt to prove, as in Proposition \ref{e-isolated left appendage}, that $f(E_1,B_0)$ or $f(B_0,E_1)$ is $e$-isolated with respect to $A$, but this assertion, if true, may be very difficult to prove. Below, we will simply explore how this problem can be analyzed through the study of certain Diophantine equations. Here are some lemmas that will aid us in this effort.

\begin{lem}
\label{lexico min preservedleft}
Let $A = f(\ldots f(E_1,B) \ldots)$ and $A' = f(\ldots f(E_2,B) \ldots)$ be $f$-expressions such that $\dege(A) = 2^p3^q = \dege(A')$, $B$ contains all the cores of $A$, and $B$ contains at least one core of $A'$. Suppose that for every occurrence $x_1$ of $x$ in $E_1$ and for every occurrence $x_2$ of $x$ in the ellipses of $A$ we have $\deg(\maxt(x_1)) > \deg(\maxt(x_2))$. Suppose that $e(E_1) <_L e(E_2)$. Then $e(A) <_L e(A')$.
\end{lem}
\begin{proof}
We can write $e(E_1)^2 = p_1(x) + q(x)$ and $e(E_2)^2 = p_2(x) + q(x)$, where $p_1(x)$, $p_2(x)$, $q(x)$ are polynomials and $\deg(p_1(x)) < \deg(p_2(x))$. Examining $(e(E_1)^2 + e(B)^3)^{\frac{2^p3^q}{3\,\dege(B)}}$ from $e(A)$ and $(e(E_2)^2 + e(B)^3)^{\frac{2^p3^q}{3\,\dege(B)}}$ from $e(A')$, we see that $(q(x) + e(B)^3)^{\frac{2^p3^q}{3\,\dege(B)}}$ is common to both $e(A)$ and $e(A')$. We have $\deg(e(A) - (q(x) + e(B)^3)^{\frac{2^p3^q}{3\,\dege(B)}}) = \deg(p_1(x)) + (\frac{2^p3^q}{3\,\dege(B)} - 1)(3\,\dege(B)) < \deg(p_2(x)) + (\frac{2^p3^q}{3\,\dege(B)} - 1)(3\,\dege(B)) \le \deg(e(A') - (q(x) + e(B)^3)^{\frac{2^p3^q}{3\,\dege(B)}})$, from which the conclusion follows.
\end{proof}

Analogously we have the following
\begin{lem}
\label{lexico min preservedright}
Let $A = f(\ldots f(B,E_1) \ldots)$ and $A' = f(\ldots f(B,E_2) \ldots)$ be $f$-expressions such that $\dege(A) = 2^p3^q = \dege(A')$, $B$ contains all the cores of $A$, and $B$ contains at least one core of $A'$. Suppose that for every occurrence $x_1$ of $x$ in $E_1$ and for every occurrence $x_2$ of $x$ in the ellipses of $A$ we have $\deg(\maxt(x_1)) > \deg(\maxt(x_2))$. Suppose that $e(E_1) <_L e(E_2)$. Then $e(A) <_L e(A')$.
\end{lem}

Suppose $\bar{A} \in \term(f;x)$ such that $e(\bar{A}) = e(A)$. Then we must have $\dege(\bar{A}) = \dege(A) = 2^p3^q$ and $\bar{A}^{[m+n]} = B_0$. Since $B_0$ contains all the cores of $A$, $B_0$ must contain all the cores of $\bar{A}$; otherwise we would have $\lead(e(\bar{A})) > \lead(e(A))$, which is a contradiction. $\bar{A}^{[m + n + 1]}$ must be either $f(E_2,B_0)$ or $f(B_0,E_2)$ for some $E_2 \in \term(f;x)$, so we have either $\bar{A} = f(\ldots f(E_2,B_0) \ldots)$ or $\bar{A} = f(\ldots f(B_0,E_2) \ldots)$. We will say that \textbf{$A$ and $\bar{A}$ have the same orientation at the $(m + n + 1)$st stage} if $A = f(\ldots f(E_1,B_0) \ldots)$ and $\bar{A} = f(\ldots f(E_2,B_0) \ldots)$, or if $A = f(\ldots f(B_0,E_1) \ldots)$ and $\bar{A} = f(\ldots f(B_0,E_2) \ldots)$. In each of the cases $A = f(\ldots f(E_1,B_0) \ldots)$ or $A = f(\ldots f(B_0,E_1) \ldots)$, our ultimate goal (which, by the way, will not be achieved in this paper) is to show that $A$ and $\bar{A}$ have the same orientation at the $(m + n + 1)$st stage and furthermore that $\dgap(E_2,B_0) = \dgap(E_1,B_0)$. In other words, in both cases we want to show that $\dege(E_2) = \dege(E_1)$, from which it would follow by Lemma \ref{lexico min preservedleft} and Lemma \ref{lexico min preservedright} that $E_2 = E_1$ whenever $e(E_1)$ is lexicographically minimal for its degree. If $\dgap(E_2,B_0) < \dgap(E_1,B_0)$, then $\deg(\maxt(E_2)) = \dege(A) - \dgap(E_2,B_0) > \dege(A) - \dgap(E_1,B_0) = \deg(\maxt(E_1))$, so under Assumption \ref{newlabel5} we have $e(\bar{A}) >_L e(A)$, which is a contradiction. Therefore, we must have \begin{equation}\label{newlabel13} \dgap(E_2,B_0) \ge \dgap(E_1,B_0). \end{equation} Now, for the remainder of this paper, we make the following 
\begin{assmpt}
\label{overarching assumption}
Either $A$ and $\bar{A}$ have the opposite orientation at the $(m + n + 1)$st stage or $\dgap(E_2,B_0) > \dgap(E_1,B_0)$.
\end{assmpt}
Notice that Assumption \ref{overarching assumption} is the negation of the statement ``$A$ and $\bar{A}$ have the same orientation at the $(m + n + 1)$st stage and $\dgap(E_2,B_0) = \dgap(E_1,B_0)$'' which we hope to eventually prove, so we will try to derive contradictions under Assumption \ref{overarching assumption}. Also notice that Assumption \ref{overarching assumption} implies that $A^{[m+n+1]}$ is \textit{not} $e$-isolated with respect to $A$. We will see in the following discussion that certain exponential Diophantine equations must hold, and we will study these exponential Diophantine equations to see how contradictions might be derived.

There exists $C \in \term(f;x)$ and positive integer $l \ge m + n + 1$ such that either $\bar{A}^{[l]} = f(C,\bar{A}^{[l-1]})$ or $\bar{A}^{[l]} = f(\bar{A}^{[l-1]},C)$, and such that $\deg(e(\bar{A}) - e(B_0)^{\frac{2^p3^q}{\dege(B_0)}}) = \deg(\maxt(C))$. Since $\deg(e(A) - e(B_0)^{\frac{2^p3^q}{\dege(B_0)}}) = \deg(\maxt(E_1))$, we must also have $\deg(e(\bar{A}) - e(B_0)^{\frac{2^p3^q}{\dege(B_0)}}) = \deg(\maxt(E_1))$, so $\deg(\maxt(C)) = \deg(\maxt(E_1))$. It follows that $\dgap(C,\bar{A}^{[l-1]}) = \dgap(E_1,B_0)$. Let us now introduce subscripts that will show more about the relation between $C$ and $B_0$. Thus, there will exist (possibly more than one choice of) $k_1, k_2 \in \mathbb{N} \cup \{0\}$ and $C_{k_1,k_2} \in \term(f;x)$ (above called $C$), such that $\bar{A}^{[m+n+k_1+k_2]}$ is either $f(C_{k_1,k_2},\bar{A}^{[m+n+k_1+k_2-1]})$ or $f(\bar{A}^{[m+n+k_1+k_2-1]},C_{k_1,k_2})$, $k_1 \ge 1$ or $k_2 \ge 1$, \begin{equation}\label{deg expression} \dege(\bar{A}^{[m+n+k_1+k_2]}) = 2^{m+k_1}3^{n+k_2}, \end{equation} and \begin{equation}\label{dgap equality} \dgap(C_{k_1,k_2},\bar{A}^{[m+n+k_1+k_2-1]}) = \dgap(E_1,B_0); \end{equation} note that, by our definition, $k_1$, $k_2$ denote the number of times $e(B_0)$ gets squared and the number of times $e(B_0)$ gets cubed, respectively, when we arrive at $\bar{A}^{[m+n+k_1+k_2]}$. We will call such a $C_{k_1,k_2}$ a \textbf{supplementing subexpression for $E_1$}; this is a generalization of the ``supplementing'' $C$ we mentioned in the paragraph following the proof of Corollary \ref{lexico min branch gen}. Notice that the case $k_1 = \pi_2$, $k_2 = \pi_1$ corresponds to the case where $A$, $\bar{A}$ have the opposite orientation at the $(m + n + 1)$st stage and where $\dgap(E_2,B_0) = \dgap(E_1,B_0)$; in this case we have $C_{k_1,k_2} = C_{\pi_2,\pi_1} = E_2$, i.e. $E_2$ is one such supplementing subexpression for $E_1$. The case $k_1 = \pi_1$, $k_2 = \pi_2$ corresponds to the case where $A$ and $\bar{A}$ have the same orientation at the $(m + n + 1)$st stage and $\dgap(C_{k_1,k_2},\bar{A}^{[m+n+k_1+k_2-1]}) = \dgap(E_2,B_0)$. Since $\dgap(E_2,B_0) > \dgap(E_1,B_0)$ by Assumption \ref{overarching assumption}, we have $\dgap(C_{k_1,k_2},\bar{A}^{[m+n+k_1+k_2-1]}) > \dgap(E_1,B_0)$, which is a contradiction. Therefore, we can exclude the case $k_1 = \pi_1$, $k_2 = \pi_2$ (which we will call the \textbf{excluded case}) in our analysis below, because this case cannot happen under Assumption \ref{overarching assumption}.

Notice that either \[2^{m+k_1}3^{n+k_2} - \dgap(C_{k_1,k_2},\bar{A}^{[m+n+k_1+k_2-1]}) = 3\dege(C_{k_1,k_2})\] or \[2^{m+k_1}3^{n+k_2} - \dgap(C_{k_1,k_2},\bar{A}^{[m+n+k_1+k_2-1]}) = 2\dege(C_{k_1,k_2}),\] and both $3\dege(C_{k_1,k_2})$, $2\dege(C_{k_1,k_2})$ must be the product of a power of 2 and a power of 3. Since $2^{m+k_1}3^{n+k_2} - \dgap(C_{k_1,k_2},\bar{A}^{[m+n+k_1+k_2-1]}) = 2^{m+k_1}3^{n+k_2} - \dgap(E_1,B_0) = 2^{m+k_1}3^{n+k_2} - (2^{m+\pi_1}3^{n+\pi_2} - 2^{i+\pi_2}3^{j+\pi_1})$, we must have \begin{equation}\label{deg equation two} 2^{m+k_1}3^{n+k_2} - (2^{m+\pi_1}3^{n+\pi_2} - 2^{i+\pi_2}3^{j+\pi_1}) = 2^{l_1}3^{l_2}, \end{equation} where $l_1, l_2 \in \mathbb{N} \cup \{0\}$. Solving (\ref{deg equation two}) will make it easier for us to determine whether or not the supplementing subexpression $C_{k_1,k_2}$ exists and more generally whether or not Assumption \ref{overarching assumption} can be true. Notice that Equation (\ref{deg equation two}) is essentially the equation \begin{equation}\label{deg equation three} 2^a3^b + 2^c3^d = 2^e3^f + 2^g3^h. \end{equation} We will study (\ref{deg equation three}) in \cite{expdiophantine}, where some partial results are proven.

\end{document}